\documentclass{article}[10 pts]
\usepackage[frenchb, english]{babel}
\usepackage{titlesec}
\usepackage{pst-node}
\usepackage{tikz-cd} 
\usepackage[T1]{fontenc}
\usepackage{geometry}
\usepackage{amsfonts}
\usepackage{amsfonts,amssymb,amsbsy,amsmath,amsthm}
\usepackage{amsmath,amscd}
\usepackage{amsmath}
\usepackage{xstring}
\usepackage{parskip}
\usepackage{graphicx}
\usepackage{mathrsfs}
\usepackage{fancyhdr} 
\usepackage{lipsum}
\usepackage{imakeidx}
\usepackage{amsthm}
\usepackage{longtable}
\usepackage{sidecap}
\usepackage[numbers,sort&compress]{natbib}
\usepackage{epstopdf}
\theoremstyle{definition}
\newtheorem{thm}{Theorem}[section]
\newtheorem{lem}[thm]{Lemma}

\theoremstyle{definition}

\theoremstyle{remark}

\graphicspath{ {./Images/.} }
\newcommand{\family} {$\mathcal F $ }
\begin{document}

   \begin{center}
        \vspace*{1 cm}
        
        \textbf{\LARGE \bf On Dynamics of $\lambda + \tan z^2 $ }
        
        \vspace{0.5 cm}
        
        \text{\Large {Santanu Nandi}}
        
 \end{center}

\vspace*{1 cm}

{\bf Abstract:} This article discusses some topological properties of the dynamical plane ($z$-plane) of the holomorphic family of meromorphic maps $\lambda  + \tan z^2$ for $ \lambda \in \mathbb C$. In the dynamical plane, we prove that there is no Herman ring, and the Julia set is a Cantor set for the maps when the parameter is in the unbounded  hyperbolic component contained in the four quadrants in the complex plane. Julia set is connected for the maps when the parameters are in other hyperbolic components of the parameter plane. \\

\vspace {1 mm}


{\bf \section{Introduction}} 

In this article we study the dynamics of the family of translated $\tan z^2 $ maps. We denote this family by $ \mathcal F$. The dynamics of the transcendental meromorphic family of maps $ \lambda \tan z^2 $  has been studied in \cite{nandi_dynamics_2020}, \cite{nandi_combinatorial_2021}. The dynamics of the family \family is much different than the dynamics of $ \lambda \tan z^2. $ 
The origin is not a fixed point for the maps in the family \family, however the origin was always a super-attracting fixed point for the maps in the family $ \lambda \tan z^2.$ The maps in \family are even, and the Fatou components of the maps in \family are symmetric about the origin. \\

For a meromorphic map, there are three possibilities for the iteration of an orbit. A point either has finite orbits (periodic points), terminates at a pole, or has infinite orbit under the iteration of the function. For a transcendental meromorphic function, for more than one pole, the pre-pole  is defined by the set $ \mathcal P = \cup_{n = 0}^{\infty} f^{\circ -n} (\infty) $.  By Picard's Theorem the set $ \mathcal P $ is infinite. We need two very important tools, called Fatou set and Julia set to study the dynamics of the family of maps. The \textit{Fatou set} or \textit{Stable set} contains the points $z$ so that $ f^{\circ n}(z) $ is defined and forms a normal family in a neighborhood of the point $z.$ The set of points that are in the complement of the Fatou set is called the \textit{Julia set}. Clearly, the Fatou set $ \mathcal S$ is open, the Julia set $ \mathcal J $ is closed and the set $ \overline {\mathcal P} \subset \mathcal J.$ So $ \mathcal J $ is non-empty. One can easily see that the Fatou and Julia set are completely invariant.\\  

We denote the \textit{singular point} $A_f$ is the set of points such that the meromorphic function $ f $ has no regular covering at the point. The value of the function at the point is called the \textit{singular value}. If a singular value is algebraic, then it is called a \textit {critical value} where as if it transcendental then it is called an \textit{asymptotic value}. For an asymptotic value $ a_f$, an \textit{asymptotic tract} is defined as the open set $ T $ such that $f(T)$ is a punctured neighborhood around $ a_f.$ The map in the family $ \mathcal F$ has two asymptotic values, $ \lambda \pm i $ and a critical value, the origin. The maps have four asymptotic tracts in each of the four quadrants.  \\

We have arranged this paper as follows. In section 2, we have discussed some basic mapping properties of the maps in \family. In section 3, we give a classification of the stable components of the maps in \family and also proved that there is no Herman rings in the stable components for these maps. In section 4, we see that there are four unbounded hyperbolic components for which the the maps has a single attracting periodic points and all singular values are attracted by that attracting basin. We briefly introduced the symbolic dynamics and proved the main result of this article that the Julia set is Cantor set for a set of parameters in the parameter plane for the maps in \family. \\

{\bf \section{ Properties of the maps in \family}}

\noindent
Consider the regions $ L_{k} = \{ z: \sqrt{\pi (k-1) + ti} < z < \sqrt{\pi k + ti}, \  t \in (-\infty, \infty), \  k \in \mathbb Z \} $. It is clearly seen that for each $ k, \  \sqrt{\pi (k-1) + ti} $ defines hyperbola in the complex plane for $ t \in (-\infty, \infty).$  Each open region bounded by two such curves, mapped to $\mathbb {\hat C}$ by $ f \in \mathcal F.$ The inverse map $ f^{-1} $ is given by the formula:

\begin{center} 
$ f^{-1} (z) = \sqrt{  \frac{1}{2i}\log \Big( \frac{1 + i(z - \lambda)}{1 -  i(z - \lambda)} \Big) } $. 
\end{center}

Hence the $ \Re f^{-1} $ will be given by $\arctan $ function. Therefore we must specify the branch of $\arctan$ when defining the inverse map. We will denote the branches as $ f_{k, \lambda}^{-1} $ in each of the region $ L_k$. \\ 


The zeros of the family of maps in $ \mathcal F $ is of the form $ \pm \sqrt{n \pi}(\pm \sqrt{n \pi} i) $ for $ n \in \mathbb Z$. We label the zeros of the family of maps on the real and imaginary axes respectively by even and odd numbers indexing. For $ n = 2 m$, the zero $p_n =  \sqrt{m \pi} $ lies on the real axis when $ m > 0$ whereas  for $ m < 0 $ and $ n = 2m $ the zero $ p_n = -  \sqrt{|m| \pi} $ lies on the negative real axis. Indexing along the imaginary axis can be done in a similar fashion for $ n = (2m + 1) $.  Thus one can find zeros $ \ldots, p_{-6}, p_{-4}, p_{-2}, p_0, p_2, \ldots $ along real axis and $ \ldots, p_{-5}, p_{-3}, p_{-1}, p_1, p_3, \ldots $ along the imaginary axis. \\

Poles of the family of maps lies on the both real and imaginary axes and are given by $ \pm \sqrt{(n + 1/2){\pi}} $ for $ n \in \mathbb Z.$ For $ m $ positive and $ n = 2m $ the pole is denoted by $ s_n = \sqrt{(m + 1/2){\pi}} $ whereas for $ m $ negative and $ n = 2m $ the pole $ s_n$ is denoted by $ s_n = - \sqrt{|m + 1/2|{\pi}}$. For $ n = (2m + 1) $, a similar indexing can be done for poles along the imaginary axis. Therefore one can find the poles of the family of maps of $ \mathcal F$  as $ \ldots, s_{-6}, s_{-4}, s_{-2}, s_0, s_2, \ldots $ along real axis and $ \ldots, s_{-5}, s_{-3}, s_{-1}, s_1, s_3, \ldots $ along the imaginary axis. \\

\begin{lem}\label{Lem1_sec_1} 
The maps and their derivatives in the family \family follow the symmetry property as below. \\
(i) $ f_{\bar{\lambda}}(\bar{z}) = \overline{f_{\lambda}(z)} .$\\
(ii) $ f_{\bar{\lambda}}'( \bar{z}) = \overline{f_{\lambda}'(z)}$ \\
\end{lem}

\begin{proof} The proofs follow from the properties of conjugates of complex numbers. \\
\end{proof}


{\bf \section {Classifications of the Fatou components }} 

Let $ S$ be a component of the Fatou set of the maps in $ \mathcal F.$ The map $ f $ maps the component $ S$ to another component of the Fatou set. If the image set contains a singular value, then the map may not be onto. In any case we call the image component $f(S)$ and that satisfy the following dichotomy: \\

(i) There exists integers $ p \neq q > 0$ such that $ f^{{\circ}(p)}(S) = f^{{\circ}(q)}(S)$. The component $ S $ is called \textit {eventually periodic}. \\
(ii) For all integers $ p \neq q $, $ f^{{\circ}(p)}(S) \cap f^{{\circ}(q)}(S) = \emptyset $, and the component $ S$ is called  \textit {wandering domain}.  \\

The qualitative and quantitative behavior of stable components of meromorphic maps are categorized by the nature of the periodic points. Suppose that $ z_0, z_1 = f(z_0), \ldots, z_p = f(z_{p-1}) = z_0 $ be a periodic cycle. We define the \textit{multiplier} or \textit{eigenvalue} of the periodic points by $ \lambda = (f^{{\circ}(p)})'(z_i), \   \  i = 0, 1, \ldots p-1 $. Suppose a stable component $ S $ eventually land on the cycle $ S_0, S_1, \ldots, S_{p-1} $. Then \\

(a) \textit {$ S_i$ is attractive}: Each domain $ S_i$ contains a periodic point of the periodic cycle and the multiplier $ |\lambda | < 1$. Some domain in the cycle must contain a singular value. If the multiplier $ \lambda = 0$ then the critical point itself is a periodic point in the cycle and the periodic cycle is called super-attractive. \\

(b) \textit{$ S_i$ is parabolic:} The multiplier map $ \lambda $ is of the form $ \lambda = e^{2 \pi i \frac{p}{q}}, \  \ (p, q) = 1$. The boundary of each component $ S_i$ contains a periodic  point of the cycle and the points in the domain $ S_i$ are attracted to the periodic point. Some component $ S_i$ must contain a singular value. \\

(c) \textit{$ S_i$ is Siegel Disk:} The multiplier map of the periodic cycle is given by $  \lambda = e^{2 \pi i \theta}, $ where $ \theta $ is an irrational number. There is a holomorphic homeomorphism from the domain $ S_i$ to the unit disk $ \mathbb D $, conjugating the first return map $ f^p$ to an irrational rotation from $ S_i$ to the unit disk $ \mathbb D $. \\

(d) \textit{$ S_i$ is Herman Ring:} Each of the periodic components $ S_i $ is holomophically homeomorphic to the standard annulus and the first return map is conjugate to an irrational rotation of the annulus by a holomorphic homeomorphism. \\

(e) \textit{$ S_i$ is essentially parabolic or Baker Domain:} The periodic domain $ S_i$ contains a point $ z_i$( possibly $\infty$) on its boundary so that $ lim_{n \to \infty} f^{{\circ}(np)}(z) = z_i $ for all $ z \in S_i $ but $ f^{{\circ}(p)}$ is not holomorphic in $ z_i$. If $ p = 1, $ the $ z_i = \infty. $ \\

From \cite{ASENS_1989_4_22_1_55_0}, one can see that the maps in the \family has no essentially parabolic domain. The next theorem shows that there is no Herman ring in the stable set. Therefore stable set can only be classified by the first three types of components as discussed above. \\  

\begin{thm} There is no Herman ring in the Fatou set of $f_\lambda \in $ \family. \\
\end{thm}

\begin{proof} Let $ X_0, X_1, \ldots, X_n $ be a cycle of Herman rings. Each of this domain $ X_i, \  i = 0, \ldots, n $ is conformally conjugate to an annulus. Let $ f^p: X_i \to X_i$ is a degree one self mapping on a domain $ X_i$.  Let $ \eta $ be a simple closed curve in $ X_i$. Then $ \eta$ must encloses a pre-pole in the bounded region, say $E_\eta$. Choosing $ f^p $ and $ \eta$  suitably, we can assume that $\eta$ passes very close to a pole and $ f $ maps the bounded region $ E_\eta $ to the complement of a bonded set. $ E_\eta $ contains a pole say, $s_k$. \\

We claim that $ E_\eta $ contains $-s_k$. If not, by symmetry, there is an $ \eta i$ in $ i E_\eta $. Therefore $ f(\eta) $ and $ -2\lambda + f(\eta i)$ both have non-zero winding number around the origin. So they must intersect, but they cannot. So $ E_\eta$ contains both $ s_k$ and $-s_k.$  This implies that the curve has a non-zero winding number around the origin and $ \eta^2 $ intersects at-least $2k$ number of vertical strips ${L_n}^2,  \  \  n = -k, -k+1, \ldots k-1, k$  defined in section 2.  Therefore the winding number of $ f^p$ is greater than $ 2k$ contradicts that $ f^p:  X_i \to X_i$ is degree one map. \\ 
\end{proof}

\begin{lem} Let $U $ be an arbitrary small neighborhood of $ \lambda + i$. Then $U$ has a non-empty intersection with  $ \mathcal J_\lambda$  if and only if $\lambda + i$ is a pre-pole of $ f_\lambda$. 
\end{lem} 

\begin{proof}
If $ \lambda + i  $ is a pre-pole of $ f_\lambda$, clearly $ \lambda + i \in \mathcal J.$  Conversely, if $ \lambda + i$ would be a point of $ \mathcal F $, then we can have a small neighborhood of $ \lambda + i$ that would be in $\mathcal F$, since $ \mathcal F $ is open. Therefore $ \lambda + i$ is in $ \mathcal J$. Any small neighborhood of $ \lambda + i $ contains a pre-pole of $ f_\lambda$ because pre-poles are dense in $ \mathcal J_\lambda.$ Let $ T_1$ be such neighborhood of $ \lambda + i$ and let $ s_1$ be a pre-pole in $\overline{T_1}$. We can have a nested closed sets of neighborhoods $\overline{T_1} \supset \overline{T_2}, \supset \overline{T_3} \supset, \ldots, \supset \overline{T_n}, \ldots $ of $ \lambda + i$ and a sequence of pre-poles $ s_i \in \overline{T_i}$ with $ s_i \neq s_j$ such that the infinite intersections of these closed nested sets is non-empty and contains a single point, namely $ \lambda + i$. Clearly the sequence $ s_n \to \lambda + i $ as $ n \to \infty.$ By proposition 6.1 of \cite {keen1995dynamics}, $ \lambda + i$ is a pre-pole of $ f_\lambda.$ \\  
\end{proof}



{\bf \section{Hyperbolic components in the parameter space}}

One can find the proof of the following theorem in \cite{10.1155/S1085337503205042}. \\
 
\begin{thm} \label{fixed point theorem} 
Let $ \mathcal T$ be a nonempty domain in a complex Banach space $Y.$ Let $s : \mathcal T \to \mathcal T$ be a bounded holomorphic map. If $ s(\mathcal T)$ lies strictly inside $\mathcal T$, then $s$ has a fixed point in $ \mathcal T$. \\
\end{thm}  

We use the above theorem in the domain of real numbers to prove the following theorem. \\

\begin{figure}
\centering
\includegraphics[height= 6 cm, width=8 cm]{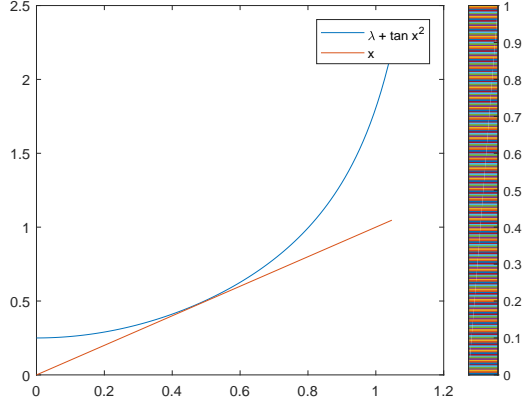} 
\caption{The map has one real fixed point for $ \lambda = \frac{1}{4}  $} 
\end{figure}

\begin{figure}
\centering
\includegraphics[height= 6 cm, width=8 cm]{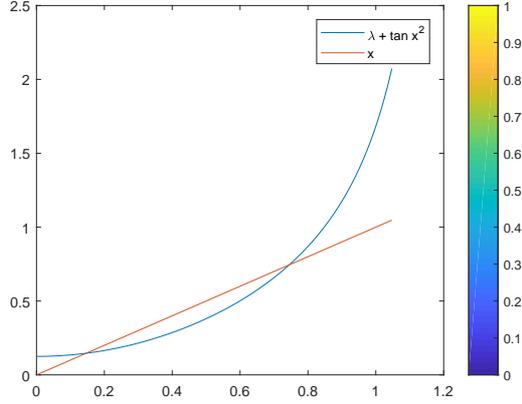}
\caption{ The map has two real fixed points for $ \lambda = \frac{1}{8} $} 
\end{figure}

\begin{thm} For $ \lambda \in \mathbb R $ and $ |\lambda| < \frac{1}{4}$, the origin is attracted by an attracting cycle. \\
\end{thm}

\begin{proof} If $\lambda$ is in $\mathbb R$, the orbit of the origin belongs to $\mathbb R$. If the function $f_\lambda$ has an attracting periodic fixed point, the attracting periodic points lie on $\mathbb R.$ To see the existence of a real fixed point of $ f_\lambda$, we prove that the map is a contraction mapping in some bounded domain of $ \mathbb R.$ The derivative $ f'_\lambda(x) $ is an increasing function in $[0, \pi/2)$ with  $f'_\lambda(0) = 0 $ By Lipschitz condition and Mean Value Theorem, there is $ x, y, c \in \mathbb R $ such that $ | f_\lambda(x) - f_\lambda(y)| \leq k |x - y| $ for $ k = |f'_\lambda(c)| < 1$.
Therefore  $f_\lambda $ has at-least a fixed point in the interval $[0, \pi/2),$ say $ \alpha$. The multiplier map for the fixed point $ \alpha$ is given by $ |f'_\lambda(\alpha)| .$ For $ \alpha$ to be an attracting fixed point, we need to have $|f'_\lambda(\alpha)| = | 2 \alpha ( 1 + \tan^2 \alpha^2)| < 1,$ implies $ |\frac{1}{2 \alpha}| > 1, $ or $ |\alpha| < 1/2.$  Therefore $ |f_\lambda(\alpha)| = |\lambda + \tan(\alpha)^2 | < 1/2, $ or $ |\lambda | < 1/4$. 

Let $\alpha$ be an attracting periodic cycle. For $|f'_\lambda(\alpha)| = 2|\alpha||1+(\alpha-\lambda)^2| < 4|\alpha^2\lambda| $ to be less than 1, we need to have $|\alpha^2 \lambda| < 1.$ This implies that for all $\lambda,$ with $|\lambda| < 1/4$ the function has an attracting fixed point $\alpha$ with $ |\alpha|^2 > 4.$  

As $ |\lambda + \tan \lambda^2 | < |\lambda| + |\tan \lambda^2| < \frac{1}{4}+ \frac{1}{4^2} $, one can check the orbit $f_\lambda^n(\lambda)$ of the critical value $\lambda$ is bounded by $ \frac{1}{4} + \frac{1}{4^{2n}} $ , i.e. $|f_\lambda^n(\lambda)| <  \frac{1}{4} + \frac{1}{4^{2n}}$. Clearly, the origin must be attracted by the attracting fixed point $\alpha.$ 
\end{proof}

\begin{figure}
\centering
\includegraphics[height= 6 cm, width=8 cm]{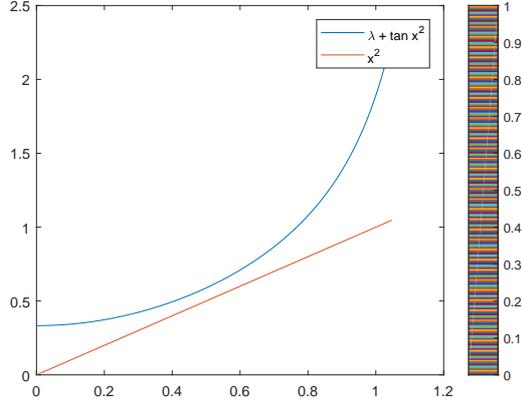} 
\caption{ The map has no fixed point for $ \lambda = \frac{1}{3} $} 
\end{figure}

\begin{lem} \label{attract_singular_values}There are four unbounded regions $ A_1, A_2, A_3, A_4 $ in the four complex quadrants such that for $ \lambda \in A_i, i = 1, 2, 3, 4,$  all singular values of $ f_\lambda$ are attracted by an attracting fixed point. \\
\end{lem}

\begin{proof} Let $ \lambda = \lambda_1 + i \lambda_2 $ and  $f_\lambda$ is holomorphic in $\lambda$. Then, $ |f_\lambda' (\lambda)| = 2 |\lambda||\sec^2 \lambda^2| \approx 4 |\lambda| e^{\pm \Im (\lambda^2)} $.

If $ \Im (\lambda^2) \to \mp \infty,$ then $ |f_\lambda' (\lambda)| \to 0.$  So there exists $ T>0$ such that for all $ \lambda$ with  $|\Im \lambda^2| > T$, $ |f_\lambda' (\lambda)| < 1 $. Consider the region $ A_i =\{ \lambda: \lambda_1 \lambda_2 > T\} , i = 1, 2, 3, 4$. So each of these regions is an unbounded region in each quadrant separated by a hyperbola $  \lambda_1 \lambda_2 > T, \     \    \lambda_1, \lambda_2 \in \mathbb R$ in the complex plane. \\

For each asymptotic value $ \lambda + i$ (or $ \lambda - i$) in $ A_i$, consider a ball $ B(\lambda + i, M) = \{ z : |z - (\lambda + i) | < M\}$ for some $M > 1$  (or $ B(\lambda - i, M)$. We will show that for each $ z \in B(\lambda + i, M),\   \   f_\lambda(z) \in B(\lambda + i, M).$ Since $ B(\lambda + i, M) $ is a Banach space, by Theorem \ref{fixed point theorem}, $ f_\lambda$ has an attracting fixed point inside the ball $ B(\lambda + i, M)$.  $ |f_\lambda(z) - (\lambda + i)|  \approx {|e^{\pm \Im(z^2)} - 1|} < M $ for suitably chosen $ z$ in $ A_i, \  i = 1, 2, 3, 4$.  Therefore, for any $ \lambda \in A_i, \  \  i = 1, 2, 3, 4,  \   \   f_\lambda $ has an attracting fixed point. However, the attracting basin of the attracting cycle contains the origin because $ \lambda $ is in $ B(\lambda + i, M) $ and $ \lambda $ is the image of the origin under $f_\lambda.$ The orbit of the other asymptotic value eventually land on the orbit of $ \lambda + i$ (or $ \lambda - i$) and hence in the corresponding attracting basin of $f_\lambda $.

\end{proof}   

\begin{thm} \label{Cantor set}
If $ \lambda \in A_i, \  i = 1, 2, 3,4 $, Julia set of $ f_\lambda $ is a Cantor set and the Julia set is topologically conjugate to shift automorphism $\sigma|_{\sum}.$ \\

\end{thm}

To prove the theorem we need to introduce the concept of shift space. Let $ \mathbb Z$ be the set of alphabet and consider the space of  infinite sequences $ \sum $ from $ \mathbb Z$ to $ \mathbb Z.$ Denote an element {\bf x} in the space $\sum $ as {\bf x} $ = ( \ldots, x_{-3}, x_{-2}, x_{-1}, x_0, x_1, x_2, x_3, \ldots ) $. Let us denote $ x_{[-k,k]} $ to be a finite word $ (x_i) $ with $ -k \leq i \leq k $. We define a metric as the following:

$$ d({\bf x, y}) =  \begin{cases} 
 		2^{-k},  & \text {if} \    k  \   \text {is the maximal such that} \  x_{[-k,k]} = y_{[-k,k]}, \\ 
		 0,  & \text{if} \  {\bf x} = {\bf y} 
\end{cases} 
$$ 		     
In other words, two points $ {\bf x} $ and $ {\bf y} $ are at $ 2^{-k}$ distance if they agree at their $ (2k +1)$-central blocks.\\ 
The topology induced by this metric is equivalent to the topology defined in \cite{Moser1973StableAR}. Define the shift map $ \sigma : \sum \to \sum $ by $ \sigma:  ( \ldots, x_{-3}, x_{-2}, x_{-1}, x_0, x_1, x_2, x_3, \ldots ) \to  ( \ldots, x_{-2}, x_{-1},  x_0, x_1, x_2, x_3, \ldots ) $. The map $ \sigma$ has the following properties. We refer to \cite{lind2021introduction} for more detail. \\

\begin{lem} 
$ \sum $ is a compact metric space. 
\end{lem}

\vspace{0.5 cm} 

\begin{lem}
$ \sigma: \sum \to \sum $ is a homeomorphism. 
\end{lem}

\vspace{0.5 cm} 

\begin{thm}
The set $ \sum $ is a Cantor set. 
\end{thm}

\vspace{0.5 cm} 

\begin{lem}
If $ \lambda \in A_i , i = 1, 2, 3, 4$, there is one multiply connected unbounded Fatou set. \\
\end{lem}
 
\begin{proof} Without loss of generality we can assume that $ \lambda $ is in the unbounded component $ A_1$ and $U$ be the Fatou component containing the attracting fixed point $p_0$ of $f_\lambda$. $ f(U) \subset U$ and $U$ must contain a singular value. Without loss of generality we can assume that $ \lambda + i$ is in $U.$ We claim that all other singular values are in $U$. Choose a sufficiently large $M>0$ with $ |\lambda| > M $ such that $ \Im (\lambda^2) > L$ for $ L>0 $ sufficiently large. Then a small neighborhood around $ \lambda + i$ contains $f_\lambda(\lambda)$. If necessary by making $ M$ large, we can get a ball $ B(\lambda + i, 1+\epsilon) $ such that $ B(\lambda + i, 1+\epsilon) \in U$. Therefore, for $ \lambda \in A_1$ with $ |\lambda| > M$ for sufficiently large $M>0$, $\lambda$ is in the the unbounded Fatou component containing the periodic point. We have that each of the $f_\lambda, f_{\lambda'}, $ for $ \lambda, \lambda' \in A_1$, are conformally conjugate and preserves the dynamics. So the singular value $ \lambda $ is in the Fatou component containing $ \lambda + i$ and the periodic fixed point for $ \lambda \in A_1$.  \\

\noindent
Define a path $ \gamma: [0,1] \to U$ such that $ \gamma(0) = \lambda $ and $ \gamma(1)= \lambda + i $.  Each branch of $f_\lambda^{-1}(\gamma) $ is an infinite path joining pre-zeros. So all pre-image $f_\lambda^{-1}(\gamma) $  must intersect the unbounded periodic Fatou component $U.$ Then $ U$ is completely invariant. Using lemma \ref{attract_singular_values}, we have that all singular values are in $U.$ 
\end{proof}

\begin{proof}
${of \   Theorem \   \ref{Cantor set}: } $ Let $P(f)$ denotes the set of post-singular points. Since all singular values are attracted to the attracting fixed point, we can enclose the post-singular points with a single set except for a finite number of points outside. Let $V'$ the set of all singular and post-singular points. Let $V = \mathbb {\hat C} \setminus V'$ be the complement of $V'$ in $\mathbb {\hat C} .$ Then $\overline V$ is a neighborhood of the points in the Julia set $\mathcal J$. All inverses of $\overline V$ under $ f^{-1} $ is well defined and have infinitely many closed bounded components. Since infinity is an essential singularity and contained in $ \overline V$, each of the bounded components contain a pole. Let $ B_n$ denotes the bounded components indexed by the pole(describe the indexing). Denote the $j-$th pre-image of $ B_n$ under $ f^{-1} $ by $ B_{\bf n}$ where $ {\bf n}$ denotes a vector defined as $ {\bf n} = (n_0, n_1, \ldots, n_j) $. Suitably choosing the branches, we can get a nested sequences of closed bounded sets, namely, $ B_{\bf n} $ such that $ B_{n_j, n_{j-1}, \ldots, n_1} \subset B_{ n_{j-1}, \ldots, n_1} \subset \ldots B_{n_1}$ for ${\bf n} = (n_j, n_{j-1}, \ldots, n_1)$. Also note that diam$(B_{\bf n}) $ tends to zero as $ j \to \infty.$ Therefore each of the set $ B_{\bf n}$ is  singleton and contains a point of the Julia set $ \mathcal J. $ Note that $ \partial B_{\bf n} \subset \mathcal F$ and $ B_{\bf n} \cap \mathcal J \neq \emptyset.$ Since $ f^N(B_{\bf n} \cap \mathcal J) = \mathcal J$, for some $N,$  we can conclude that $ \mathcal J$ is a totally disconnected set and $ \mathcal J = \cup_{{n_j} = 1}^{\infty} \cap_{j = 1}^{\infty} B_{\bf n} $.  \\

Assign a point $ \overline {s}  = ( s_1, s_2, \ldots, ) $ to a point $ z $ in the set $ \mathcal J $ defined by $ {\bf s} = \phi(z) $ such that $ z \in B_{\overline s} $, i.e. $ z \in B_{s_1}\cap B_{(s_1, s_2)}\cap \ldots $. In other word, the point $z$ should also satisfy $ f^{n-1}(z) \in B_{(s_1, s_2, \ldots, s_n)} $. For a point $z$ with $f^m(z) = \infty, $  we consider $ f^{m-1}(z) \in B_{(s_1, s_2, \ldots, s_m, \infty )} $. We need to show that $\phi $ is a bijection. Claim that $ \phi(z_1) \neq \phi(z_2) $ if $ z_1 \neq z_2.$ If not, there is $ z_1 \neq z_2$ such that $ \phi(z_1) = \phi(z_2) .$ There exist $\bf {s} $ and $ \bf t$ such that $ {\bf {s}} = \phi(z_1) = \phi(z_2) = {\bf {t}} $. Let $ \bf {s_i}$ denotes a sequence defined as  ${ \bf {s_i}} = ( s_1, s_2, \ldots, s_i) $ and $ g_{\bf m_j} $ denotes an inverse branch of $ f^j$. There exist two inverse branches $  g_{\bf s_i}, g_{\bf t_j} $ so that $ g_{\bf s_i}(B_{\bf s_i})  \cap g_{\bf t_j}(B_{\bf t_j}) = \emptyset$ with $ g_{\bf s_i} (f^i(z_1)) = z_1$( and $ g_{\bf t_j} (f^j(z_2)) = z_2$ ). Since $ B_{\bf s_i} \cap P(f) = \emptyset $ and $ B_{\bf t_j} \cap P(f) = \emptyset $ we can analytically extend these inverse branches to the set $V.$ Consider the branches  $  g_{s_i}, g_{ t_j} $ of $ f^{-1} $ so that $ g_{s_i}(f^i(z_1)) = f^{i-1}(z_1) $ and $ g_{t_j}(f^j(z_2)) = f^{j-1}(z_2) $. By assumption both $ f^{i-1}(z_1) $ and $ f^{j-1}(z_2) $ are in $ B_{\bf t_j} $ implies that $  g_{s_i}= g_{ t_j} $. Inductively all other branches of  $f^{-1}$ are also equal and that implies $ z_1 = z_2. $ For $ {\bf s_i} = ( s_1, s_2, \ldots, \infty) $ and $ {\bf t_j} = ( t_1, t_2, \ldots, \infty ) $ we have that $ z_1 = g_{s_1} \circ g_{s_2} \circ \ldots \circ g_{s_i} (\infty) = g_{t_1} \circ g_{t_2} \circ \ldots \circ g_{t_j} (\infty) = z_2.$ Hence the map $ \phi(z) $ is an injection. \\

\begin{figure}
\centering
\includegraphics[width=\textwidth]{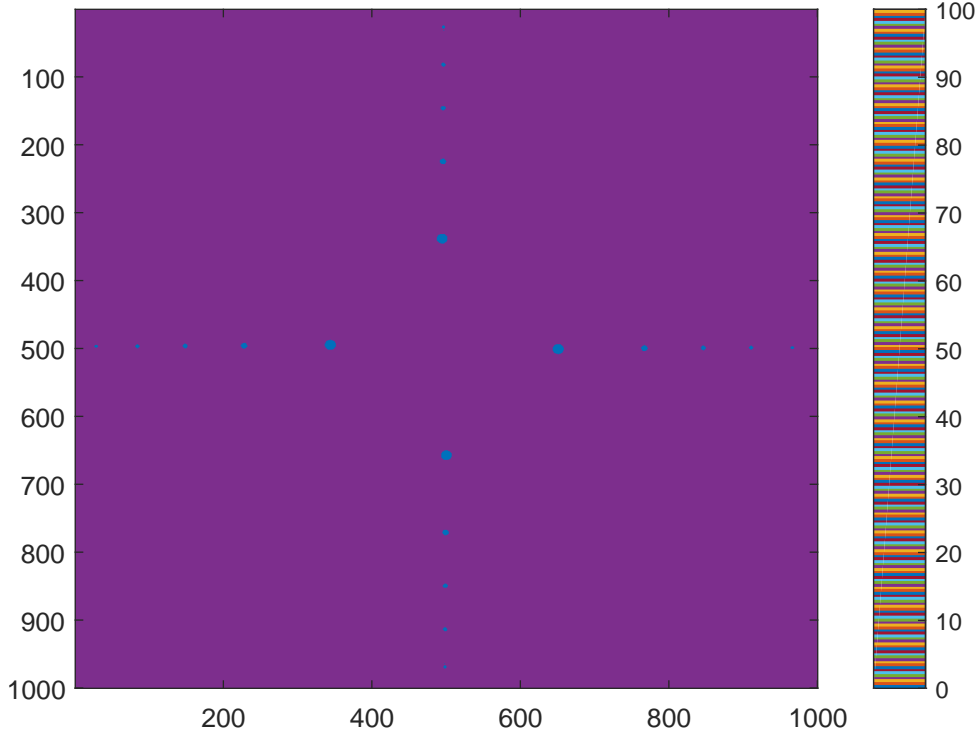} 
\caption{ Julia set is a Cantor set for $ \lambda = 4 + 4i $} 
\end{figure}

Let ${ \bf w } = ( w_1, w_2, \ldots, )$ be an element in the sequence space $ \sum $. Let's consider $ {g_{\bf w}'} $ be the inverse branches of $f$. For each of the inverse branch $ g_{w_i}' $ on $ V$ consider the set $ N_{w_i} = g_{w_i}' (V) \cap \mathcal J.$ Then each of the set $ N_{w_i} $ is closed bounded set and the diameter tends to zero. Now choose the branches $ g_{w_i} $ of $ f^{-1} $ on $  N_{w_i} $ in such a way that $   g_{w_i}(N_{w_i}) = N_{w_{i+1}} \subset N_{w_i} $. So we have $ g_{w_1} \circ g_{w_2} \ldots  \circ g_{w_n}(V) \subset  g_{w_1} \circ g_{w_2} \ldots  \circ g_{w_{n-1}}(V) \subset  g_{w_1}(V) $. This is a sequence of nested closed sets with diameter of the sets tends to zero. Therefore the set $ \cap_{i = 1}^{\infty} N_{w_i} $ is singleton and contains a point $ z $ of $ \mathcal J.$ So $ \phi(z) = {\bf w}.$ If $ { \bf w } = ( w_1, w_2, \ldots, w_k, \infty) $, then $ g_{w_1} \circ g_{w_2} \ldots \circ g_{w_k}(\infty) $ is a prepole and a single point. Hence the map $ \phi $ is onto. \\

Since $ \sum$ and $ \mathcal J $ are compact and $ \phi $ is a continuous bijection, the inverse map $ \phi^{-1} $ is also  continuous. \\

If $ z_0$ is a point in $ \mathcal J $, there is a $ \bf t $ in $ \sum $ so that $ \phi(z_0) = \bf t.$ That implies $ f^{n-1} (z_0) \in B_{(t_1, t_2, \ldots, t_n)} $. Hence the map $ \sigma(\phi(z_0)) = \bf {t+1} $ shifts the element of $ \bf t $ one unit to the right. On the other way, $ \bf t $ is uniquely determined by $ \phi $ for the point $ z_0 \in \mathcal J.$ So $ f(z_0) \in B_{(t_1)}. $ Therefore the map $ \phi(f(z_0)) $ determines the point $ \bf t $ shifted one unit to the right, i.e. $ \bf {t+1}.$ Hence the following diagram is commutative.

$$
\begin{tikzcd}
\mathcal J \ar{d}{\phi} \ar{r}{f}          & \mathcal J  \ar{d}{\phi} \\
\sum  \ar{r}{\sigma} & \sum
\end{tikzcd}
$$ 

Thus $ \mathcal J $ is a Cantor set. 

\end{proof}

Pictures of the parameter space give nice topological and combinatorial structure of the hyperbolic components. It will be interesting to see the bifurcation of the periodic cycle on the boundary of the bounded hyperbolic components. We plan to explore this further in our future work. \\

\pagebreak

\bibliography{mybib1}{}



\end{document}